\documentclass[ paper=a4, fontsize=11pt,bibliography=totoc, final]{scrartcl}
\usepackage{amsthm, amssymb, amsmath, amsfonts, mathrsfs, dsfont}

\usepackage[utf8]{inputenc}
\usepackage[T1]{fontenc}
\usepackage{lmodern}
\usepackage{graphicx,graphics}
\usepackage{mathtools}
\usepackage[english]{babel}
\usepackage{csquotes}
\usepackage{epsfig,url}  
\usepackage{bbm}
\usepackage{a4wide}
\usepackage{enumerate}
\usepackage{verbatim} 
\usepackage{color}
\usepackage{esint}
\usepackage{microtype}

\usepackage[T1]{fontenc} 

\usepackage[colorlinks=true, pdfstartview=FitV, linkcolor=blue, citecolor=blue, urlcolor=blue,pagebackref=false]{hyperref}
\usepackage[notcite,notref,color]{showkeys}
\definecolor{labelkey}{gray}{.8}
\definecolor{refkey}{gray}{.8}

\definecolor{darkblue}{rgb}{0,0,0.7} 
\definecolor{darkgreen}{rgb}{0,0.5,0}

\newcommand{\ud}{\;\mathrm{d}} 
\newcommand{\uud}{\mathrm{d}}

\providecommand{\Fbeta}{F_{\beta,c_\beta}}

\newtheorem{theorem}{Theorem}[section]

\newtheorem{proposition}[theorem]{Proposition}

\newtheorem{lemma}[theorem]{Lemma}

\newtheorem{assumption}[theorem]{Assumption}

\numberwithin{equation}{section}
\numberwithin{theorem}{section}

\newcommand{\Reals}{{\mathbb R}}

\newcommand{\R}{{\mathbb R}}

\renewcommand{\varrho}{{\rho}}

\newcommand{\N}{\mathbb N}

\newcommand{\loc}{\mathrm{loc}}

\newcommand{\step}[1]{\noindent \textit{Step} #1.}

\usepackage[margin=3cm,includefoot]{geometry}

\mathtoolsset{showonlyrefs}

\usepackage{authblk}

\setkomafont{date}{\large}

\title{Existence and non-uniqueness of stationary states for the Vlasov-Poisson equation on $\R^3$ subject to attractive background charges }{}

\author[1]{ Raphael Winter\thanks{raphael.elias.winter@univie.ac.at}}
\affil[1]{University of Vienna, Austria}

\begin{document}
	
	\maketitle

	\begin{abstract}
			We prove the existence of stationary solutions for the density of an infinitely extended plasma interacting with an arbitrary configuration of background charges. Furthermore, we show that the solution cannot be unique if the total charge of the background is attractive. In this case, infinitely many different stationary solutions exist. The non-uniqueness can be explained by the presence of trapped particles orbiting the attractive background charge.
	\end{abstract}

	\tableofcontents

\section{Introduction}

We consider the response of a spatially homogeneous plasma to a given background distribution of charges. This phenomenon can be modeled by the following nonlinear stationary Vlasov-Poisson equation for the plasma electron density $f(x,v)$, on the three dimensional phase space $\R^3\times \R^3$
\begin{align}
    v \cdot \nabla_x f - \nabla_x Q\cdot \nabla_v f &= 0, \label{eq:Vlasov1}\\
    -\Delta Q &= \rho[f]-1 + \mu, \label{eq:Vlasov2}  \\
    \lim_{|x|\rightarrow \infty} f(x,v) &= f_0(v). \label{eq:Vlasov3}
\end{align}
Here $\mu$ describes the distribution of background charges and $f_0$ the electron distribution far away from the perturbation.
As usual, $\rho[f]$ denotes the spatial density of electrons given by
\begin{align}\label{eq:rho} 
    \rho[f](x) = \int_{\Reals^3} f(x,v) \ud{v}.
\end{align}
The system~\eqref{eq:Vlasov1}-\eqref{eq:Vlasov3} is often considered in plasma physics, since the onset of Debye screening can quickly be derived for the linearized system. A detailed discussion can be found in the plasma physics textbooks~\cite{goldston_introduction_1995} and \cite{nicholson_introduction_1983}. The model~\eqref{eq:Vlasov1}-\eqref{eq:Vlasov3} is also used in other contexts in plasma physics, for instance for characterizing plasma waves (cf.~\cite{schamel_stationary_1971}). For a plasma interacting with a repulsive point charge, screening has been proved rigorously in~\cite{arroyo-rabasa_debye_2021}.

A number of results have been shown for the nonlinear Vlasov-Poisson equation in the case of finite mass and finite energy. In this case, the conserved quantities can be used to study existence and stability of stationary solutions (cf.~\cite{rein_non-linear_1994}). We also refer to recent results on the stability of a point charge interacting with a plasma of finite mass ~\cite{pausader_stability_2021,pausader_stability_2022}. 

A closely related problem is the existence and stability of stationary solutions to the Vlasov-Poisson Boltzmann equation. In the presence of collisions, the natural boundary condition $f_0$ are Maxwellian distributions.  We refer to~\cite{duan_optimal_2011,duan_stability_2010,duan_existence_2007} for details.

The key point of this paper is to show that stationary solutions to~ \eqref{eq:Vlasov1}-\eqref{eq:Vlasov3} exist for general background measures $\mu$, and infinitely many stationary states $f$ exist as soon as the total charge of the background measure is attractive. 

The main result is contained in the following theorem.

\begin{theorem} \label{thm:main}
    Let $f_0(v) = F_0(\tfrac12 |v|^2)$ for some function $F_0\in C^1(\R)$ satisfying Assumption~\ref{ass:F0} below, and $\mu\in \mathcal{M}(\R^3)$ be a measure with finite total variation, i.e.
    \begin{align}\label{eq:mu} 
        \int_{\R^3} |\mu|(\uud{x}) < \infty.
    \end{align}
    Then there exist a  solution $f\in W^{1,1}_{\loc}(\R^3\times \R^3)\cap L^1_{\loc}(\R^3;L^1(\R^3))$, $Q\in W^{1,1}(\R^3)\cap L^2(\R^3)$ to the stationary Vlasov-Poisson problem~\eqref{eq:Vlasov1}-\eqref{eq:Vlasov3}.  Here the equation for $Q$ is understood in the weak sense, and the boundary condition~\eqref{eq:Vlasov3} as 
    \begin{align}
        \| f(\cdot,v)-f_0(v)\|_{L^2(\R^3)}< \infty, \quad v\in \R^3 \, a.e.
    \end{align}
    If the total charge $\theta\in \Reals$ given by
    \begin{align} \label{eq:theta} 
        \theta= \int_{\R^3} \mu(\uud{x})
    \end{align}
    is negative, then there exist infinitely many different solutions to the stationary problem~\eqref{eq:Vlasov1}-\eqref{eq:Vlasov3}.
\end{theorem}
We split the proof of Theorem~\ref{thm:main} in parts. The existence of solutions is shown in Proposition~\ref{prop:existence}, using the extensions $F$ of $F_0$ constructed in Section~\ref{sec:extension}.  The non-uniqueness of solutions is given by Proposition~\ref{prop:uniqueness}.

In~\cite{arroyo-rabasa_debye_2021}, the existence of stationary solutions and their screening properties have been investigated for 
\begin{align}
    \mu(\uud{x}) = \theta \delta_0(\uud{x}), \quad \theta >0 . 
\end{align}
The existence proof in~\cite{arroyo-rabasa_debye_2021} relies on  radial symmetry of the constructed solution $f$ and the compact embedding $H^1_r(\Reals^3) \subset L^p_r(\R^3)$, $2<p<6$ under radial symmetry. Most importantly, the proof only applies to repulsive interaction, i.e. $\theta>0$.

The remainder of the paper is devoted to the proof of Theorem~\ref{thm:main}. Some parts of the proof follow similar to ~\cite{arroyo-rabasa_debye_2021}. New ideas are needed to deal with general, non-radial solutions, the attractive case $\theta<0$ and non-uniqueness of solutions.

As in~\cite{arroyo-rabasa_debye_2021}, we look for solutions $f$ of the form
\begin{align}\label{eq:Form} 
    f(x,v) = F\big(\tfrac12 |v|^2 + Q(x)\big).
\end{align}
Hence, if $Q(x)\rightarrow 0$ as $|x|\rightarrow \infty$, then the boundary condition in~\eqref{eq:Vlasov3} yields
\begin{align}
    f_0(v) = F(\tfrac12 |v|^2).
\end{align}
In the case of a repulsive point charge, we can show $Q\geq 0$. 
Therefore, the function $F$ in~\eqref{eq:Form} is completely determined by the boundary condition $f_0$ in~\eqref{eq:Vlasov3}. In the presence of an attractive test charge, $Q$ also attains negative values and the function $F$ is not uniquely determined by $f_0$. Physically, this can be explained by the presence of electrons trapped on orbits around the attractive background charge. This phonomenon is also discussed in~\cite{schamel_stationary_1971}.

In order to prove that infinitely many solutions exist for $\theta<0$ in~\eqref{eq:theta}, we construct a family of admissible extensions $F$ of $F_0$
\begin{align}\label{eq:F}
    F(r) = \begin{cases} F_0(r) \quad &\text{if }r\geq 0,\\
                        \tilde{F}(r) \quad &\text{if }r<0.
    \end{cases} 
\end{align}     
The set of admissible extensions $F$ is characterized by the function
\begin{align} \label{eq:g} 
    g(r) = 4 \pi \sqrt{2} \int_0^\infty \sqrt{s} F(r+s) \ud{s}, \quad r\in \Reals.
\end{align}

We make the following assumption on the distribution $f_0$ of the plasma at $|x|\rightarrow \infty$, which is also assumed in~\cite{arroyo-rabasa_debye_2021}.

\begin{assumption}[Velocity distribution at infinity] \label{ass:F0}
    The boundary condition $f_0$ can be represented as $f_0(v) = F_0\big(\frac12 |v|^2\big)$, where $F_0\in C^1(\Reals^+;\R^+)$ satisfies
    \begin{enumerate}[(i)]
        \item normalization
            \begin{align}\label{eq:F0normal}
                 4 \pi \sqrt{2} \int_0^\infty \sqrt{r} F_0(r) \ud{r} = 1.
            \end{align}
        \item decay condition
        \begin{align}\label{decay}
            |F_0(r)| + |F_0'(r)|\leq \frac{C}{1+r^3}. 
        \end{align}
        \item stability condition: $F_0$ satisfies: 
            \begin{align}\label{eq:monotone} 
                F_0'(r)<0,\quad r\geq 0. 
            \end{align}
    \end{enumerate}
\end{assumption}

For future reference, we define
\begin{align} \label{eq:sigma} 
    \sigma = -g'(0) >0. 
\end{align}

\section{Extension to the negative half-line}  \label{sec:extension}

We show the existence of solutions if the function $g$ defined in~\eqref{eq:g} satisfies the following four properties:
\begin{enumerate}[(i)]
    \item normalization:
        \begin{align}\label{cond:g0} 
            g(0)=1. 
        \end{align}
    \item differentiability and monotonicity: $g\in C^2(\R)$ and  
        \begin{align} \label{cond:g1}
            g'(r)< 0,\quad r\in \Reals. 
        \end{align}
    \item sub-differential at zero: 
        \begin{align} \label{cond:g2}
             g(r) \geq ( g(0) + g'(0) r).
        \end{align}
    \item growth condition: for some $1<\alpha<\frac32$ we have
    \begin{equation} \label{cond:g3}
    \begin{aligned}
        g(r) -( g(0) + g'(0) r)&\leq C_1  |r|^\alpha, \\
         |g'(r)-g'(0)| &\leq C_2 |r|^{\alpha-1}.
    \end{aligned}
    \end{equation}
\end{enumerate}

Before we establish the existence of solutions under the conditions ~\eqref{cond:g0}-\eqref{cond:g3}, we demonstrate that it is possible to extend the function $F_0$ to the negative half-line such that these conditions are met. This is the content of the following lemma. 

\begin{lemma} \label{lem:Fbeta} 
Let $F_0$ satisfy Assumption~\ref{ass:F0}.
    For $\beta\in (0,\frac12)$ and $c_\beta>0$ consider the function
\begin{align} \label{eq:Fbeta} 
        \tilde F_{\beta,c_\beta}(r) =   \frac{c_\beta r^2}{ \langle r\rangle^{\beta+2}} + e^{-r^2} (F_0(0) +F_0'(0) r)  \quad r\leq 0,
\end{align} 
 and let $F_{\beta,c_\beta}$ be the extension of $F_0$ by $\tilde{F}_{\beta,c_\beta}$. Here $\langle r\rangle = \sqrt{1+|r|^2}$ is the japanese bracket.
 
 Then $F_{\beta,c_\beta}\in C^1_b(\R;\R^+)$ and for $c_\beta>0$ large enough, the function $g_{\beta,c_\beta}$ defined by~\eqref{eq:g} satisfies the conditions~\eqref{cond:g0}, \eqref{cond:g1}, \eqref{cond:g2} and \eqref{cond:g3} with $\alpha=\frac32-\beta $. 
\end{lemma}
\begin{proof}
    \step1 By construction we have $F_{\beta,c_\beta} \in C^1_b(\R)$. Moreover, 
    $F_{\beta,c_\beta}>0$ is positive since $F_0'>0$ (cf.~\eqref{eq:monotone}).
    
    \step2 The normalization condition~\eqref{cond:g0} follows since $F$ is an extension of $F_0$ and $F_0$ satisfies~\eqref{eq:F0normal}. 
    Furthermore, the derivative of $g_{\beta,c_\beta}$ can be represented as
    \begin{align}
        g_{\beta,c_\beta}'(y) = -4\pi \sqrt{2} \int_0^\infty \frac{\Fbeta(r+y)}{\sqrt{r}} \ud{r}. 
    \end{align}
    Since $\Fbeta$ is positive,~\eqref{cond:g1} follows.  
    
    \step3 For the proof of~\eqref{cond:g2}, we first remark that the condition holds for $r\geq0$. As observed in~\cite{arroyo-rabasa_debye_2021} this follows since $F_0$ satisfies~\eqref{eq:monotone}, and therefore $g_{\beta,c_\beta}$ is 
    \begin{align}
        g_{\beta,c_\beta}''(y) = -4\pi \sqrt{2} \int_0^\infty \frac{\Fbeta'(r+y)}{\sqrt{r}} \ud{r}>0 , \quad y\geq 0,
    \end{align}
    convex on the positive half-line. 
    
    \step 4 We  prove that~\eqref{cond:g2} holds for $r<0$ if $c_\beta>0$ large enough. 
     To this end, we decompose $g_{\beta,c_\beta}$ for $r\leq 0$ into 
    \begin{align}\label{eq:gbeta}
        g_{\beta,c_\beta}(r)    &= g_{\beta,0}(r) + c_\beta  4 \pi \sqrt{2} \int_0^{|r|} \frac{\sqrt{s}|r+s|^2}{(1+|r+s|^2)^{\frac{\beta+2}2}} \ud{s}.
    \end{align}
    We now observe that there exists $c'>0$ small enough such that~\eqref{cond:g2} is satisfied for $r\in [-c',0]$, independent of $\beta$, $c_\beta$. This follows from  $g''_{\beta,0}(0)>0$, and the other contribution in~\eqref{eq:gbeta} being positive. On the other hand, for $r\leq -c'$, we can estimate the second term in~\eqref{eq:gbeta} below by
    \begin{align}
        4 \pi \sqrt{2} \int_0^{|r|} \frac{\sqrt{s}|r+s|^2}{(1+|r+s|^2)^{\frac{\beta+2}2}} \ud{s}\geq c(c') c_\beta |r|^{\frac32-\beta}. 
    \end{align} 
    Since $\beta\in (0,\frac12)$, we can choose $c_\beta>0$ large enough such that~\eqref{cond:g2} holds.
    
    \step4 Since $g_{\beta,c_\beta}\in C^2(\R)$, the condition~\eqref{cond:g3} holds locally. It remains to check the asymptotics for $r\rightarrow -\infty$. We again use~\eqref{eq:gbeta}. From the decay condition on $F_0$~\eqref{decay} we easily obtain
    \begin{align}
        |g_{\beta,0}(r)|&\leq C|r|^\alpha,\\
        |g'_{\beta,0}(r)|&\leq C|r|^{\alpha-1}.
    \end{align}
    Similarly, the estimate follows for the integral term in~\eqref{eq:gbeta} by a straightforward computation.
    
\end{proof}

\section{Existence of solutions} 

\begin{lemma}\label{lem:B}
    Let $F\in C^1_b(\R)$ be a non-negative function such that the function $g$ (cf.~\eqref{eq:g}) satisfies the properties~\eqref{cond:g0}-\eqref{cond:g3} for some $1 <\alpha<\frac32$. Recall $\sigma>0$ introduced in~\eqref{eq:sigma} and for $P\in L^p(\R^3)$ define 
    \begin{align}\label{eq:BP}
        B[P](x) = g(P(x)) -1 + \sigma P(x). 
    \end{align}
    Then  we have 
    \begin{align}\label{eq:B2alpha} 
       0\leq  B[P](x) \leq C_0 (|P(x)|^\alpha \wedge |P|^2). 
    \end{align}
    Moreover, $B$ is a continuous operator $B: L^{2\alpha}(\R^3)\rightarrow L^2(\R^3)$.
\end{lemma}
\begin{proof}
    The non-negativity of $B$ follows from the subdifferential condition~\eqref{cond:g2} on $g$. For the upper bound, we first recall that
    $g\in C^2$, and $\sigma$ is defined by~\eqref{eq:sigma}. Since $\alpha<2$ this yields
    \begin{align}
        0\leq B[P] \leq C(|P|^\alpha \wedge |P|^2),\quad  |P|\leq 1.
    \end{align}
    For $|P|\rightarrow \infty$ the inequality follows from the growth condition~\eqref{cond:g3}.

    Continuity of the operator $B:L^{2\alpha}(\R^3) \rightarrow L^2(\R^3)$ follows from 
    \begin{align}
        \|B[P]-B[Q]\|_{L^2(\R^3)} &\leq \| (|P|+|Q|)^{\alpha-1} |P-Q|\|_{L^2(\R^3)} \\
        &\lesssim \left( \|P\|_{L^{2\alpha}(\R^3)}+\|Q\|_{L^{2\alpha}(\R^3)}\right) \|P-Q\|_{L^{2\alpha}(\R^3)},
    \end{align}
    and finishes the proof.
\end{proof}

\begin{proposition}[Existence of solutions]\label{prop:existence}
    Let $F\in C^1_b(\R)$ be an extension of $F_0$ such that the function $g$ (cf.~\eqref{eq:g}) satisfies the properties~\eqref{cond:g0}-\eqref{cond:g3} for some $1 < \alpha < \frac32  $.
    Then there exists a solution  $f\in W^{1,1}_{\loc}(\R^3\times \R^3)\cap L^1_{\loc}(\R^3;L^1(\R^3))$, $Q\in W^{1,1}(\R^3)\cap L^2(\R^3)$ to the stationary Vlasov-Poisson system~\eqref{eq:Vlasov1}-\eqref{eq:Vlasov3} in the sense of Theorem~\ref{thm:main}.
\end{proposition}
\begin{proof}
    \step1  Recall $\sigma>0$ defined in~\eqref{eq:sigma}, and let $\Phi_\sigma$ be the fundamental solution to $(\sigma-\Delta)^{-1}$, i.e.
    \begin{align}\label{eq:PhiSigma}
        \Phi_\sigma(x) = \frac{e^{-\sqrt{\sigma}|x|}}{4\pi |x|}. 
    \end{align}
    Define $S\in S'(\R^3)$ by
    \begin{align}\label{S:Lp}
        S = \Phi_\sigma * \mu .
    \end{align}
    Since $\mu\in \mathcal{M}(\R^3)$ has finite total variation~(cf.~\eqref{eq:mu}), we know
    \begin{align}
        \|S\|_{L^p(\R^3)} < \infty, \quad p\in[1,3),
    \end{align}
    and in light of~\eqref{eq:B2alpha}, $B[S]$ satisfies
    \begin{align}
        \|B[S]\|_{L^p(\R^3)} < \infty, \quad p\in [1,\frac3{\alpha}).
    \end{align}
    By the Green's function property of $\Phi_\sigma$, $S$ is a weak solution to
    \begin{align} \label{iden:S} 
        (\sigma-\Delta) S = \mu.
    \end{align}
    Let us further introduce the functions $H_1$, $H$ by
    \begin{align}
        H_1(x)  &= \phi * B[S] ,\\
        H(x)    &=  C_0 ((|S|+|H_1|)^\alpha \wedge (|S|+|H_1|)^2) ,
    \end{align}
    where $C_0>0$ is the constant appearing in~\eqref{eq:B2alpha} and 
    \begin{align} \label{eq:phiriesz}
    \phi(x) = \frac{1}{4 \pi |x|}.
    \end{align}
    Using that $\phi$ defined in~\eqref{eq:phiriesz} is a Riesz potential, we find that
    \begin{align}
        \|H_1\|_{L^p(\R^3)} &< \infty, \quad p\in (3,\infty),\\
        \|H\|_{L^2(\R^3)}    &< \infty. \label{est:H}
    \end{align}
    \step2 
    We pick the coefficient $q_0$ as
    \begin{align}\label{eq:q0} 
        q_0= 2 \alpha  , 
    \end{align}
    and define the operator
    \begin{align}
        K: L^{q_0}(\Reals^3)    &\rightarrow L^{q_0}(\Reals^3) \\
            R               &\mapsto \Phi_\sigma * \left(B(R+S  )\wedge H\right) .
    \end{align}
    We claim that $K$ is a continuous compact operator. Continuity follows from the continuity of $B$ shown in Lemma~\ref{lem:B}. It remains to prove that the image of $K$ is precompact. To this end, consider a sequence $R_k\in L^{q_0}(\R^3)$, and $G_k := K[R_k]$. Then for some $M>0$ we have 
    \begin{align}
        0 \leq G_k &\leq \mathcal{H}:= \Phi_\sigma * H ,\label{eq:upper} \\
        \|G_k\|_{W^{1,q_0}(\R^3)} &< M, \label{eq:W1}
    \end{align}
    with $\Phi_\sigma$ as introduced in~\eqref{eq:PhiSigma}. We observe that due to~\eqref{est:H} the upper bound $\mathcal{H}$ satisfies
    \begin{align} \label{eq:cH} 
        \|\mathcal{H}\|_{L^p(\R^3)} = \|\Phi_\sigma * H\|_{L^p(\R^3)}  < \infty , \quad 2\leq p\leq \infty .
    \end{align}
    Compactness now follows from the Riesz criterion for precompactness in $L^p(\R^3)$ since
    \begin{align}
        \sup_{k\in \N} \|G_k\|_{L^{q_0}(\R^3)} &<\infty, \\
        \sup_{k\in \N} \|G_k\|_{L^{q_0}(\R^3\setminus B_R)} &\rightarrow 0 \quad \text{as } R\rightarrow \infty,\\
        \sup_{k\in \N} \|G_k(\cdot-h)-G_k(\cdot)\|_{L^{q_0}(\R^3)} &\rightarrow 0 \quad \text{as } |h|\rightarrow 0.  
    \end{align}
    For the first and second condition, we use~\eqref{eq:upper} and $\mathcal{H}\in L^{q_0}(\R^3)$, the last line follows from~\eqref{eq:W1}.

    \step3 We infer from Step~2 the existence of  a non-negative solution $R\in W^{1,q_0}(\R^3)\cap L^2(\R^3)$ to the equation
    \begin{align}\label{fixed:K}
        R = K(R)  . 
    \end{align}
    This follows  from Schaefer's fixed point theorem, since $K$ maps into a bounded set in $L^{q_0}(\R^3)$, so in particular there exists $M>0$ such that for $P\in L^{q_0}(\Reals^3)$ and $\lambda\in [0,1]$ with $P = \lambda K(P) $ we have
    \begin{align}
        \|P\|_{L^q(\Reals^3)} < M.
    \end{align}
    Together with Step~1 and Schaefer's fixed point theorem, this allows us to conclude the existence of a fixed point to the mapping $K$, i.e. $R$ satisfying~\eqref{fixed:K}. Non-negativity of $R$ follows by construction of $K$.
    
    \step4 There exists $R\in L^{q_0}(\R^3)$ such that
    \begin{align}\label{fixed:B}
        R= \Phi_\sigma * B(R+S) .
    \end{align}
    Notice that~\eqref{fixed:B} follows from~\eqref{fixed:K} if we can show $B[R+S]\leq H$. To this end, let $R\in L^{q_0}(\R^3)$ be any solution to \eqref{fixed:K}. Then $R$ is a weak solution to 
    \begin{align}
        (\sigma -\Delta)R = B[R+S] \wedge H. 
    \end{align}
    Unpacking the definition of $B$ (cf.~\eqref{eq:BP}) yields
    \begin{align}
        (\sigma -\Delta)R = \left( g(R+S)-1 +\sigma(R+S) \right) \wedge H,
    \end{align}
    so in particular 
    \begin{align}
        -\Delta R = \left( g(R+S)-1 +\sigma S \right) \wedge (H-\sigma R).
    \end{align}
    Since $R\geq 0 $ and $g$ is monotone decreasing due to~\eqref{cond:g1}, we infer
    \begin{align}
        -\Delta R \leq g(S) -1 + \sigma S = B[S],
    \end{align}
    and therefore $R\leq \phi * B[S]=H_1$. Courtesy of~\eqref{eq:B2alpha} we conclude $B[R+S] \leq H$.
    From $R$ satisfying~\eqref{fixed:B} we also infer $R\in W^{1,1}(\R^3)$.
    
    \step5 We define $Q$ by
    \begin{align}
        Q= R+ S.
    \end{align}
    Then $Q\in W^{1,1}(\R^3) \cap L^2(\R^3)$ is a weak solution to
    \begin{align}\label{iden:Q}
        -\Delta Q = g(Q) - 1 + \mu .
    \end{align}
    Here we have used the equation~\eqref{iden:S} and~\eqref{fixed:B}. 
    We now define the non-negative function $f$ by
    \begin{align}
        f(x,v) = F(\tfrac12 |v|^2 + Q(x)) .
    \end{align}
    Since $F\in C^1_b(\R)$, we obtain $f\in W^{1,1}_\loc(\R^3)$ and $f$ satisfies~\eqref{eq:Vlasov1}. The spatial density of $f$ simplifies to
    \begin{align}\label{eq:rhorep}
        \rho[f](x) = \int_{\R^3} f(x,v) \ud{v} = \int_{\R^3} F(\tfrac12 |v|^2+ Q(x)) \ud{v} = g(Q(x)). 
    \end{align}
    Hence $f\in L^1_\loc(\R^3;L^1(\R^3))$ and combining~\eqref{iden:Q} and~\eqref{eq:rhorep} we conclude~\eqref{eq:Vlasov2}. 
    For the boundary condition~\eqref{eq:Vlasov3} we use Lipschitz continuity of $F\in C^1_b(\R)$ to estimate
    \begin{align}
        \|f(\cdot,v)-f_0(v)\|_{L^2(\R^3)} &=\|F(\tfrac12 |v|^2+Q(\cdot))-F(\tfrac12 |v|^2)\|_{L^2(\R^3)} \\
                &\lesssim \|Q\|_{L^2(\R^3)} < \infty,
    \end{align}
    where we have used~\eqref{eq:upper},~\eqref{eq:cH} and~\eqref{S:Lp}.
\end{proof}

\section{Non-uniqueness}

In the preceeding sections we have constructed solutions of the Vlasov-Poisson equation of the form
\begin{align}
    f(x,v) = F(\tfrac12 |v|^2 + Q(x)),
\end{align}
for an infinite class of extensions $F_\alpha$ of the function $F_0$ to the negative half-line. 
The following Lemma proves that the associated solutions $f_\alpha$ do not coincide .

\begin{proposition}[Non-Uniqueness] \label{prop:uniqueness}
    Let $\beta_1<\beta_2\in (0,\frac12)$ and $F_1=F_{\beta_1,c_\beta}$, $F_2=F_{\beta_2,c_\beta}$, with $c_\beta>0$ large enough, the functions constructed in Lemma~\ref{lem:Fbeta}. 
    
    Let further $f_1,Q_1$ and $f_2,Q_2$ be associated solutions to the system~\eqref{eq:Vlasov1}-\eqref{eq:Vlasov3} provided by Proposition~\ref{prop:existence}. Then we have
    \begin{align}\label{eq:nonunique} 
        f_1\neq f_2 \in W^{1,1}_\loc(\R^3),\quad Q_1\neq Q_2 \in W^{1,1}(\R^3).
    \end{align}
\end{proposition}
\begin{proof}
    \step1 We first prove  $Q_1(x)<0$ on a set of positive measure. Recall the equation for $Q$
    \begin{align}
        -\Delta Q = g(Q) - 1 + \mu \in \mathcal{M}(\R^3).
    \end{align}
    Assume $Q \geq 0$ a.e.. Then $g(Q)-1\leq 0$, and therefore
    \begin{align}
        0 = \int_{\R^3} -\Delta Q(x) \ud{x} \leq \int_{\R^3} \mu(\uud{x}) = \theta < 0, 
    \end{align}
    which yields a contradiction.
    
    \step2 The functions $Q_1$ and $Q_2$ do not coincide, i.e. $|Q_1-Q_2|\neq 0 \in L^1_{\loc}(\R^3)$. Again we argue by contradiction. If $Q_1=Q_2$ a.e. then
    \begin{align}
        g_1(Q_1(x)) = g_2(Q_1(x))\quad  x\in \R^3 a.e..
    \end{align}
    However, we have $g_1(r)>g_2(r)$ on the negative half-axis $r<0$ since $\beta_1<\beta_2$. Since $Q_1$ takes negative values on a set of positive measure we reach a contradiction. 
    
    \step3 Finally $|f_1-f_2|\neq 0\in L^1_\loc(\R^3)$. Assume the contrary, and let $A\subset \R^3$ the set on which $Q_1<0$. By Step~1, $A$ has positive measure. Then we estimate 
    \begin{align}
        0   &= \int_{\R^3\times \R^3} |f_1(x,v)-f_2(x,v)| \ud{x} \ud{v}\\
            &\geq  \int_{A} \int_{\R^3} |F_1(\tfrac12 |v|^2 + Q_1(x))  - F_2(\tfrac12 |v|^2 + Q_2(x))|  \ud{v} \ud{x}\\
            &\geq  \int_{A} \inf_{r\in \R}\int_{\R^3} |F_1(\tfrac12 |v|^2 + Q_1(x))  - F_2(\tfrac12 |v|^2 + r)|  \ud{v} \ud{x}.
    \end{align}
    On the other hand, for any $y<0$ the integrand 
    \begin{align}
        \inf_{r\in \R} \int_{\R^3} |F_1(\tfrac12 |v|^2 + y)  - F_2(\tfrac12 |v|^2 + r)|\ud{v} >0 
    \end{align}
    is positive. This leads to a contradiction to the assumption. 
\end{proof}

	\section*{Acknowledgements}
	
	R.W. acknowledges support of SFB 65 "Taming Complexity in Partial Differential Systems" at the University of Vienna. Furthermore, R.W. would like to thank
	the Isaac Newton Institute for Mathematical Sciences for support and hospitality during the programme
	"Frontiers in kinetic theory: connecting microscopic to macroscopic scales - KineCon 2022" when work
	on this paper was undertaken. This work was supported by EPSRC Grant Number EP/R014604/1.
	
	\section*{Data Availability Statement}
Data sharing is not applicable to this article as no new data were created or analysed in this study.
\section*{Author Declarations}
The author has no conflicts to disclose.

\nocite*{}	
	
\bibliographystyle{plain}
\bibliography{Attractive_Screening}
	
\end{document}